\documentclass[10pt]{article}

\usepackage{amsmath}
\usepackage{amsthm}
\usepackage{amssymb}
\usepackage{verbatim}
\usepackage[dvips]{graphicx}
\usepackage{enumerate}
\usepackage{t1enc}
\usepackage{url}
\usepackage{a4wide}
\usepackage[latin2]{inputenc}
\usepackage[english]{babel}
\usepackage{color}
\usepackage[margin=6mm,left=3cm,right=3cm,top=2.5cm,bottom=2.5cm,nohead]{geometry}

\theoremstyle{plain}
\newtheorem{theorem}{Theorem}
\newtheorem{lemma}{Lemma}

\theoremstyle{definition}

\newtheorem{conjecture}{Conjecture}
\theoremstyle{remark}

\theoremstyle{definition}
\title{Note on the Diameter of Path-Pairable Graphs}
\author{G\'abor M\'esz\'aros\footnote{Research was supported by the Balassi Institute, the Fulbright Commission, and the Rosztoczy Foundation.} \\Central European University \\ \url=Meszaros_Gabor@ceu-budapest.edu=}
\begin{document}
\maketitle
\begin{abstract}
A graph on $2k$ vertices is path-pairable if for any pairing of the vertices the pairs can be joined by edge-disjoint paths. The so far known families of path-pairable graphs have diameter of at most 3. In this paper we present an infinite family of path-pairable graphs with diameter $d(G)=O(\sqrt{n})$ where $n$ denotes the number of vertices of the graph. We prove that our example is extremal up to a constant factor.
\end{abstract}
\section*{Introduction}
Given a fixed integer $k$, a graph $G$ on at least $2k$ vertices is {\it $k$-path-pairable} if for any pair of disjoint sets of vertices $X=\{x_1,\dots,x_k\}$ and $Y=\{y_1,\dots,y_k\}$ of $G$ there exist $k$ edge-disjoint paths $P_i$ such that $P_i$ is a path from $x_i$ to $y_i$, $1\leq i\leq k$. The path-pairability number of a graph $G$ is the largest positive integer $k$ for which $G$ is $k$-path-pairable. The motivation of setting edge-disjoint paths between certain pairs of nodes naturally arose in the study of communication networks. There are various reasons to measure the capability of the network by its path-pairability number, that is, the maximum number of pairs of users for which the network can provide separated communication channels without data collision (see \cite{Cs} for additional details). The nodes corresponding to the users are often called {\it terminal nodes} or {\it terminals}. A graph $G$ on $n=2m$ vertices is {\it path-pairable} if it is $m$-path-pairable, that is, for every pairing of the vertex set $\{x_1,y_1, \dots, x_m,y_m\}$ there exist edge-disjoint paths $P_1,\dots,P_m$ joining $x_1$ to $y_1$,\dots,$x_m$ to $y_m$, respectively. By definition, path-pairable graphs are $k$-path-pairable for $0\leq k\leq \lfloor\frac{n}{2}\rfloor$.

The three dimensional cube $Q_3$ and the Petersen graph $P$ are both known to be path-pairable. The graph shown in Figure 1 is the only path-pairable graph with maximum degree 3 on 12 or more vertices.
Apart from such small and rather sporadic examples we only know few path-pairable families.
Certainly, the complete graph $K_{2k}$ on $n=2k$ vertices is path-pairable. It can be proved easily that the complete bipartite graph $K_{a,b}$ is path-pairable as long as $a+b$ is even and $a,b\neq 2$. Particular species of the latter family, the star graphs $K_{2a+1,1}$ show that path-pairability is achievable even in the presence of vertices of small degrees. They also illustrate that vertices of large degrees are easily accessible transfer stations to carry out linking in the graphs without much effort. That motivates the study of $k$-path-pairable and path-pairable graphs with small maximum degree. Faudree, Gy\'arf\'as and  Lehel  \cite{3reg} gave  examples of $k$-path-pairable graphs for every $k\in\mathbb{N}$ with maximum degree 3. Note that their construction has exponential size in terms of $k$ and is not path-pairable.
\begin{figure}[]\label{pic}
\begin{center}
\includegraphics[width=3cm]{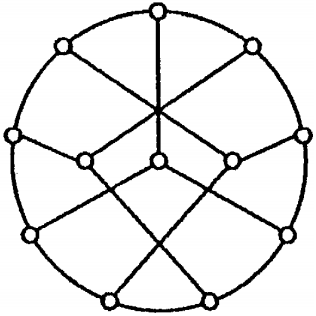} \hspace{1cm}
\caption{Path-pairable graph of order 12}
\end{center}
\end{figure}
Unlike in case of $k$-path-pairability, the maximum degree $\Delta(G)$ must increase together with the size of a path-pairable graph $G$. Faudree, Gy\'arf\'as and Lehel \cite{pp} proved that if $G$ is a path-pairable graph on $n$ vertices with maximum degree $ \Delta$ then $n\leq 2\Delta^\Delta$. The theorem gives an approximate lower bound of $\frac{\log(n)}{\log\log(n)}$ on $\Delta(G)$. By contrast, the graphs of the above presented families have maximum degree $\frac{n}{2}$ or more. Kubicka, Kubicki and Lehel \cite{grid} investigated path-pairability of complete grid graphs and proved that the two-dimensional complete grid $K_a\times K_{b}$ of size $n=ab$ is path-pairable. For $a=b$ that gives examples of path-pairable graphs with maximum degree $\Delta=2a-2<2\sqrt{n}$. In the same paper they raised the question about similar properties of three-dimensional complete grids. Note that if path-pairable, the grid $K_m\times K_m\times K_m$ yields an even better example of size $n=m^3$ and maximum degree $\Delta=3(m-1)=O(\sqrt[3]{n})$. 

We mention that one of the most interesting and promising path-pairable candidate is the $n$-dimensional hypercube $Q_n$. $Q_1=K_2$ is path-pairable while $Q_2=C_4$ is not as pairing of the nonadjacent vertices of the cycle cannot be linked. One can prove that the cube $Q_3$ is path-pairable, the question for higher odd-dimensions has yet to be answered (it was proved in \cite{F} that $Q_n$ is not path-pairable for $n$ even).  
\begin{conjecture}[\cite{Cs}]\label{q}
The $(2k+1)$-dimensional hypercube $Q_{2k+1}$ is path-pairable for all $k\in\mathbb{N}$.
\end{conjecture}
A common attribute of the known path-pairable graphs is their small diameter. For each pair $(x,y)$ of terminals in the examples above, the length of the shortest $x,y$ paths is at most 3. While terminal pairs of an actual pairing may not always be joined by shortest paths, small diameter gives the advantage of quick accessibility of the vertices and makes designation of edge-disjoint paths easier. The question concerning the existence of an infinite family of path-pairable graphs with unbounded diameter naturally arises. We use the notation 
$d(G)$ for the diameter of the graph $G$ and $d_{max}(n)$ for the maximal diameter of a path-pairable graph on $n$ vertices. We mention that if true, Conjecture \ref{q} proves the lower bound $\log_2 n\leq d_{max}(n)$ for $n=2^{2k+1}$, $k\in\mathbb{N}$.

The main goal of this note is to study the largest possible diameter of paith-pairable graphs. 
We present a family of path-pairable graphs $\{G_n\}$  such that $G_n$ has $n$ vertices and diameter $O(\sqrt{n})$ for infinitely many values of $n$. We show that our construction is optimal up to a constant factor by proving the following theorem.

\begin{theorem}\label{up}
If $G$ is a path-pairable graph on $n$ vertices with diameter $d\geq 20$ then $d\leq 6\sqrt{2}\cdot \sqrt{n}$.
\end{theorem} 

For a subgraph $H$ of a graph $G$, $|H|$ denotes the number of vertices in $H$.  The degree of a vertex $x$ and the distance of  vertices $x$ and $y$ are denoted by $d(x)$ and $d(x,y)$, respectively. For additional details on path-pairalbe graphs we refer the reader to \cite{Cs}, \cite{F} and \cite{pp}.  

%%%%%%%%%%%%%%%%%%%%%%%%%%%%%%%%%%%%%%%%%%%%%%%%%%%%%%%%%%%%%%%%%%%%%%%%%%%%%%%%%%%%%%%%%%%%%%%%%%%%%%%%%%%%%%

%%%%%%%%%%%%%%%%%%%%%%%%%%%%%%%%%%%%%%%%%%%%%%%%%%%%%%%%%%%%%%%%%%%%%%%%%%%%%%%%%%%%%%%%%%%%%%%%%%%%%%%%%%%%%%

%%%%%%%%%%%%%%%%%%%%%%%%%%%%%%%%%%%%%%%%%%%%%%%%%%%%%%%%%%%%%%%%%%%%%%%%%%%%%%%%%%%%%%%%%%%%%%%%%%%%%%%%%%%%%%

\section*{Construction}

We construct our example of a path-pairable graph on $n$ vertices with diameter $O(\sqrt{n})$ by the graph operation called "blowing-up".
%In a blown-up graph $H$ of a graph $G$ we substitute each vertex of $G$ by a class of independent vertices and join two vertices of different classes by an edge if the  corresponding vertices of $G$ are joined.
%As an example, any complete $n$-partite graph
%$K_{a_1,\dots,a_n}$ can be derived from the complete graph $K_n$ on vertex set $\{x_1,\dots, x_n\}$ by blowing up the vertex $x_i$ to an independent set of size $a_i$ for $i=1,2,\dots, n$.
%In the current construction

Set an arbitrary pairing of the vertices of $G$. We accomplish the linking of the pairs in two phases. During the first phase, for each pair of terminals, we define a path that starts at one of the terminals and ends at some vertex in the class of its pair. If the ending vertex happens to be the actual pair of the terminal, we set this path as the joining path for the given pair, otherwise we continue with the second phase. If two terminals initially belong to the same class, then the pair simply skips the first phase of the linking.
Direct our cycle $C_{2m}$ and the blown-up graph $G$ counterclockwise and label each pair $x,y$ such that there exists a directed $x\rightarrow y$ path of length at most $m$. We start building the above mentioned path for pair $(x,y)$ at vertex $x$. Fix $m$ edge-disjoint matchings $M_{1}^i,\dots,M_{m}^i$ of size $(4m+3)$ between every consecutive classes $S_i$ and $S_{i+1}$. For a pair of terminals $(x,y)$ lying in classes $S_i$ and $S_{i+d}$ (modulo $2m$) at distance $d$ ($1\leq d\leq m$), choose the edge of $M_{1}^{i}$ being adjacent to $x$ and label the other vertex adjacent to it by $p_1(x)$. In step $j$ for $2\leq j\leq d$ take the edge of $M_{j}^{i+j-1}$ being adjacent to $p_{j-1}(x)$ and label its other end by $p_{j}(x)$. Apparently, $y$ and $p_d(x)$ belong to the same class. Phase one ends by assigning a path $P_{xy} =
xp_1(x)\dots p_d(x)$ to each $(x,y)$ pair of terminals.

Observe that paths $P$ and $P'$ assigned to and starting at terminals $x$ and $x'$ of the same class do not contain a common vertex as they are given edges of the same matchings in every step. Now assume that edge $e=(x_{i,j}, x_{i+1,k})$ has been utilized by two paths $P_1$ and $P_2$. It means that $e\in M_{t}^i$ for some $1\leq t\leq m$ and that $P_1$ and $P_2$ must have started in the same class. However, in order to share an edge they also have to share a vertex which contradicts our previous observation. It proves that phase one terminates without edge-collision.

In phase two we finish the linking. For the terminal $y$ initially paired with $x$ and for the endpoint $p_d(x)$ of the path $P_{xy}$ (both vertices lying in $S_i$ for some $i$) consider the yet unused edges of the bipartite subgraph $H_i$ spanned by $S_i$ and $S_{i+1}$. As $d_{H_i}(y),d_{H_i}(p_d(x))\geq 3m+3$, there exists at least $2m+3$ different vertices $z_1,z_2,\dots,z_{2m+3}\in S_{i+1}$ such that $(y,z_k), (p_d(x),z_k)\in E(G), k=1,2,\dots,(2m+3)$. Observe that any vertex in $S_i$ is an endpoint of at most $m+1$ paths defined in phase one, hence out of the $(2m+3)$ listed candidates at most $(2m+2)$ could have been assigned to another pair $(y',p_{d'}(x'))$. It means $x$ and $y$ can be joined by the path $xp_1(x)\dots p_d(x)z_iy$ with an appropriate choice of $z_i$ from the list above. That completes the proof.

%%%%%%%%%%%%%%%%%%%%%%%%%%%%%%%%%%%%%%%%%%%%%%%%%%%%%%%%%%%%%%%%%%%%%%%%%%%%%%%%%%%%%%%%%%%%%%%%%%%%%%%%%%%%%%

%%%%%%%%%%%%%%%%%%%%%%%%%%%%%%%%%%%%%%%%%%%%%%%%%%%%%%%%%%%%%%%%%%%%%%%%%%%%%%%%%%%%%%%%%%%%%%%%%%%%%%%%%%%%%%

%%%%%%%%%%%%%%%%%%%%%%%%%%%%%%%%%%%%%%%%%%%%%%%%%%%%%%%%%%%%%%%%%%%%%%%%%%%%%%%%%%%%%%%%%%%%%%%%%%%%%%%%%%%%%%

\section*{Proof of Theorem \ref{up}}
Assume $G$ is a paith-pairable graph on $n$ vertices with diameter $d\geq 20$. Let $x,y\in V(G)$ such that $d(x,y)=d$.
Define $S_i=\{z\in V(G): d(x,z)=i\}$ for $i=0,1,2,\dots,d$ and $U_i=\bigcup\limits_{j=0}^iS_j$. Set the notation $s_i=|S_i|$ and $u_i=|U_i|$. Note that there is no edge between any $S_i$ and $S_j$ ($i<j$) classes unless they are consecutive, that is, $j=i+1$. Also, we may assume without loss of generality that vertices belonging to consecutive $S_i$ sets are joined by an egde (adding edges between consecutive classes changes neither the diameter nor the path-pairability property of the graph). If $S_d$ contains vertices in addition to $y$, move them to $S_{d-1}$ by joining them to every vertex of $S_{d-2}$ as well as to $y$ in case they were not adjacent. Observe, that in our new graph $z\in S_i$ if and only if $d(y,z)=d-i$, $0\leq i\leq d$. Note again that, while our operation may change the distribution of the vertices among the $S_i$ classes, our newly obtained graph $G'$ has the same diameter as $G$ and is also path-pairable. We could even assume every $S_i$ class to be a complete graph (by adding the necessary edges) without changing the mentioned properties, though this would be a rather cosmetic operation that would simplify the structure of the graph but would not essentially contribute to the proof.

We introduce the notation $S'_i$ for $S_{d-i}$ and divide our sets into three parts creating left, middle and right segments $A,B$ and $C$ as follows:
\begin{center}
\begin{tabular}{c c c c c}
$A=\bigcup\limits_{i=0}^{\lceil\frac{d}{3}\rceil}S_i$&	\hspace{2cm}	&$B=\bigcup\limits_{i=\lceil\frac{d}{3}\rceil+1}^{\lfloor\frac{2d}{3}\rfloor-1}S_i$ &	\hspace{2cm}	&$C=\bigcup\limits_{i=\lfloor\frac{2d}{3}\rfloor}^dS_i$ \\ 
\end{tabular} 
\end{center}
Our main goal is to give a lower estimate on the size of $B$. We first prove the following lemmas:
\begin{lemma}\label{l}
 $s_{2k}+s_{2k+1}\geq k$ as long as $u_{2k+1}\leq \frac{n}{2}$.
\end{lemma}
\begin{proof}We prove our statement by induction on $k$. Apparently, $s_0+s_1\geq 0$ and $s_2+s_3\geq 1$. Observe that the number of edges between $S_{2k+2}$ and $S_{2k+3}$ is at least $u_{2k+1}=\sum\limits_{i=0}^{2k+1}s_i$. Indeed, placing $u_{2k+1}$ terminals in $U_{2k+2}$ and their pairs in $V(G)-U_{2k+2}$ there must be space for at least $u_{2k+1}$ edge-disjoint paths passing from $S_{2k+2}$ to $S_{2k+3}$. By induction hypothesis $u_{2k+1} \geq \frac{1}{2}k(k+1)$ while the number of edges between the two classes is at most $s_{2k+2}s_{2k+3}\leq \big(\frac{s_{2k+2}+s_{2k+3}}{2}\big)^2$.  It yields $\frac{1}{2}k(k+1)\leq \big(\frac{s_{2k+2}+s_{2k+3}}{2}\big)^2$, that is, $s_{2k+2}+s_{2k+3}\geq \sqrt{2k(k+1)}\geq k+1$ if $k\geq 1$.
\end{proof}
\begin{lemma}
$|A|,|C|\geq min\big(\frac{n}{2},\frac{d^2}{100}\big)$.
\end{lemma}
\begin{proof}
Assume $|A|<\frac{n}{2}$. Using Lemma \ref{l} we know that 
$|A|=u_{\lceil\frac{d}{3}\rceil}=s_0+\dots+s_{\lceil\frac{d}{3}\rceil}\geq 0+1+\dots+\big\lfloor\frac{\lceil\frac{d}{3}\rceil}{2}\big\rfloor\geq\frac{d^2}{100}$. By exchanging the role of $x$ and $y$ the same reasoning shows that $|A'|\geq \frac{d^2}{100}$ if $|A'|<\frac{n}{2}$, where 
$A'=\bigcup\limits_{i=0}^{\lceil\frac{d}{3}\rceil}S'_i$. As $\lceil\frac{d}{3}\rceil+\lfloor\frac{2d}{3}\rfloor=d$, one can easily see that $C=A'$ which completes the proof.
\end{proof}

If $\lceil\frac{d}{3}\rceil<t<\lfloor \frac{2}{3}d\rfloor$,  the number of edges between $S_t$ and $S_{t+1}$ is at least  $min\big(\frac{n}{2},\frac{d^2}{100}\big)$. As seen before, placing $ min\big(\frac{n}{2},\frac{d^2}{100}\big)$ terminals in $A$ and their pairs in $C$ the set $S_t$ has to be able to bridge  $ min\big(\frac{n}{2},\frac{d^2}{100}\big)$ disjoint paths to $S_{t+1}$. The number of crossing edges between these two sets is at most $s_t\cdot s_{t+1}$. It means $\frac{s_t+s_{t+1}}{2}\geq \sqrt{s_ts_{t+1}}\geq  min\big(\frac{\sqrt{n}}{\sqrt{2}},\frac{d}{10}\big)$. 
If $d\geq 16$, then $\lfloor\frac{2}{3}d\rfloor-\lceil\frac{d}{3}\rceil-2\geq \frac{d}{3}-\frac{10}{3}\geq 2$ and $B$ contains at least two different $S_i$ sets. That gives us the requested lower bound on $|B|$:
\begin{equation}
|B|\geq\sum\limits_{t=\lceil\frac{d}{3}\rceil+1}^{\lfloor \frac{2}{3}d\rfloor-2}\frac{s_t+s_{t+1}}{2}\geq\big(\frac{d}{3}-\frac{10}{3}\big)min\big(\frac{\sqrt{n}}{\sqrt{2}},\frac{d}{10}\big)\geq \frac{d}{6}min\big(\frac{\sqrt{n}}{\sqrt{2}},\frac{d}{10}\big) 
\end{equation}
As $|B|\leq n$, our equation proves that $d\leq 6\sqrt{2}\sqrt{n}$.

\section*{Summary}
We proved that the maximal diameter of a path-pairable graph on $n$ vertices has order of magnitude $O(\sqrt{n})$, that is, $\frac{1}{4}\sqrt{n}<d_{max}(n)<6\sqrt{2}\sqrt{n}$ if $n$ is sufficiently large. We mention that in Lemma \ref{l} an even stronger result can be proved about  $s_{2k}+s_{2k+1}$. 
\begin{lemma}\label{comp}
 Given $\varepsilon>0$, $s_{2k}+s_{2k+1}\geq (2-\varepsilon)\cdot k$ holds for sufficiently large $k$, as long as $u_{2k+1}\leq \frac{n}{2}$.
\end{lemma}
This result with more careful calculation and elaborate analysis may help reducing the gap between the presented upper and lower bounds.
 
\section*{Acknowledgement}
The author wishes to thank Professor Ervin Gy\H ori for his  helpful suggestions, interest, and guidance. The author would also like to thank the anonymus referees for their valuable comments and corrections.

\end{document}